\documentclass{amsart}

\usepackage{amsfonts}       
\usepackage{amsthm}
\usepackage{amssymb}
\usepackage{graphicx}

\usepackage{amsthm}

\newtheorem{theorem}{Theorem}[section]
\newtheorem{corollary}{Corollary}[theorem]
\newtheorem{lemma}[theorem]{Lemma}

\numberwithin{equation}{section}

\begin{document}

\title[Shape sensitivity of the Hardy constant]{Shape sensitivity of the Hardy constant involving the distance from a boundary submanifold}

\author{M. Paschalis}
\address{Department of Mathematics, University of Athens}

\email{mpaschal@math.uoa.gr}
\thanks{This research was supported by the Hellenic Foundation for Research and Innovation (HFRI) under the HFRI PhD Fellowship grant (Fellowship Number 1250)}

\subjclass[2000]{26D15; 35P15}

\keywords{Hardy inequality, Best constant, Domain perturbation, Submanifold}

\begin{abstract}
We investigate the continuity and differentiability of the Hardy constant with respect to perturbations of the domain in the case where the problem involves the distance from a boundary submanifold. We also investigate the case where only the submanifold is deformed.
\end{abstract}

\maketitle

\section{Introduction}

Suppose $\Omega \subset \mathbb{R}^n$ is a bounded domain (open, connected) with boundary $\partial \Omega$, and let $\Sigma \subset \partial \Omega$ be a submanifold of the boundary of dimension $\dim \Sigma = s \in \{ 0,\ldots, n-1 \}$. If there exists a positive constant $C>0$ such that the inequality
\begin{equation}
\int_\Omega |\nabla u|^2 dx \geq C \int_\Omega \frac{u^2}{d_\Sigma^2} dx, \ \ \ u\in H^1_0(\Omega),
\end{equation}
with $d_\Sigma = \textrm{dist}(\cdot,\Sigma)$ is valid, we say that the Hardy inequality is satisfied for the pair $(\Omega,\Sigma)$. Such inequalities are known to hold in a variety of settings, and for the particular cases $s=0$ (point singularities) and $s=n-1$ (the entire boundary) the relevant literature is quite extensive (especially for the later case).

An important aspect of the theory of Hardy inequalities is to specify the best constant for a particular pair $(\Omega,\Sigma)$, that is, the quantity
\begin{equation}
H(\Omega,\Sigma) = \inf_{u\in H^1_0(\Omega)} \frac{\int_\Omega |\nabla u|^2 dx}{\int_\Omega \frac{u^2}{d_\Sigma^2} dx},
\end{equation}
which is called the Hardy constant.

The particular case $s=n-1$ has been studied extensively and, apart from the well known  convexity condition, other conditions have been introduced which guarantee that the Hardy constant assumes the critical value $1/4$, see \cite{D, MMP, A, G, BFT1} and references therein.

The case $s=0$ has also been studied in recent years. Sufficient conditions to have $H(\Omega,\{ \sigma \})= n^2/4$ (again, the critical value) have been obtained in \cite{BFT2, F}, see also \cite{C, CV, FMu} for results within this context. For the intermediate dimension case see \cite{FM}.

In this paper, we are primarily concerned with the behaviour of this constant under perturbations of the domain and the submanifold. In particular, if $\varphi$ is a diffeomorphism, we get a map
\begin{equation}
\varphi \longmapsto H(\varphi(\Omega),\varphi(\Sigma)),
\end{equation}
and our task is to investigate questions of continuity and differentiability of that map in an appropriate sense which is made precise in the next section. This problem has already been studied in a more general $L^p$ setting for the special case $\Sigma = \partial \Omega$ in \cite{BL}, so our work here is a natural continuation of that work.

We also concern ourselves with the problem where only the submanifold is perturbed. This is expressed in a very neat way in the case of a point singularity: if we regard the Hardy constant as a function $H:\partial \Omega \rightarrow \mathbb{R}$, $$H(\sigma) = H(\Omega, \{ \sigma \}),$$ then this function is differentiable on $\partial \Omega$, under some reasonable assumptions.

\section{Diffeomorphism Groups}

In this section we offer a quick review of finite order diffeomorphism groups in $\mathbb{R}^n$. For details, see \cite{Ba}. A $C^k$-diffeomorphism of $\mathbb{R}^n$ is a homeomorphism $\varphi: \mathbb{R}^n \rightarrow \mathbb{R}^n$ that is $k$-times bi-differentiable. The set of all such maps is denoted by $Diff^k(\mathbb{R}^n)$. It is obviously a group under composition. For our purposes, it is sufficient to work with the subgroup $Diff^k_c(\mathbb{R}^n)$ of $C^k$-diffeomorphisms with compact support (the support of a diffeomorphism is defined to be the closure of the set of points that the diffeomorphism acts upon non-trivially). Since we work on bounded domains, this is done without loss of generality, and spares us some technical considerations that are consequence of the non-compactness of $\mathbb{R}^n$.

We now equip $Diff^k_c(\mathbb{R}^n)$ with the weak $C^k$ topology (or compact-open topology). To describe this topology, it suffices to describe the basic open sets that generate it. These are the ``balls'' $$\mathcal{N}_\varphi(K,\epsilon) = \{ \psi \in Diff_c^k(\mathbb{R}^n): \| \psi - \varphi \|_{C^k(K)} < \epsilon \}$$ of center $\varphi \in Diff^k_c(\mathbb{R}^n)$, radius $\epsilon>0$ and domain $K$, which is a compact subset of $\mathbb{R}^n$. Here, we assume $$\| \varphi \|_{C^k(K)} = \sum_{0\leq |\alpha| \leq k} \| \partial^\alpha \varphi \|_{L^{\infty}(K)}.$$ In this topology, $Diff^k_c(\mathbb{R}^n)$ is a topological group, which is in fact locally homeomorphic to the Banach space of $C^k$ vector fields of compact support $\mathfrak{X}^k_c(\mathbb{R}^n) \cong C^k_c(\mathbb{R}^n,\mathbb{R}^n)$, thus assuming the structure of an infinite dimensional Lie group.

The directional derivative of a continuous function $H:Diff^k_c(\mathbb{R}^n) \rightarrow \mathbb{R}$ at $\varphi \in Diff^k_c(\mathbb{R}^n)$ in the direction of $\xi \in \mathfrak{X}^k_c(\mathbb{R}^n)$ is given by the limit $$D_\varphi H (\xi) = \frac{d}{dt} \bigg|_{t=0} H(\varphi+t\xi),$$ provided it exists. Note that the compact support assumption guarantees that $\varphi + t\xi$ is always a diffeomorphism provided that $t$ is small enough. If this is defined for all $\varphi \in Diff^k_c(\mathbb{R}^n)$ and all $\xi \in \mathfrak{X}^k_c(\mathbb{R}^n)$, we say that $H$ is (Gateaux) differentiable.

\section{Continuity of the Hardy Constant}

Here we discuss some continuity results. By $co(\Omega)$ we denote the convex hull of $\Omega$.

\begin{theorem}
Let $\Omega \subset \mathbb{R}^n$ be an open set with non-empty boundary, and let $\Sigma \subset \partial \Omega$ be an arbitrary subset of the boundary. Then there exist $\epsilon>0$ and $c>0$ such that for every $C^1$ diffeomorphism $\varphi$ with $\|D\varphi-I \|<\epsilon$,
\begin{equation}
\label{cont}
|H(\varphi(\Omega),\varphi(\Sigma)) - H(\Omega, \Sigma)| \leq c H(\Omega,\Sigma) \| D\varphi-I \|_{L^{\infty}(co_\varphi(\Omega))},
\end{equation}
where $co_\varphi(\Omega)=co(\Omega)\cup \varphi^{-1}(co(\varphi(\Omega)))$.
\end{theorem}

\begin{proof}
Let $u\in H^1_0(\Omega)$ be normalised by $\int_\Omega u^2/d_\Sigma^2 dx = 1$. For $v=u\circ \varphi^{-1}$, consider the Rayleigh quotient $$R(\varphi(\Omega),\varphi(\Sigma))[v] = \frac{\int_{\varphi(\Omega)} |\nabla v|^2 dy}{\int_{\varphi(\Omega)} v^2/d_{\varphi(\Sigma)}^2 dy} = \frac{\int_\Omega |(D\varphi)^{-\top} \nabla u|^2 |\det D\varphi| dx}{\int_\Omega \frac{u^2}{d_{\varphi(\Sigma)}^2 \circ \varphi} |\det D\varphi| dx},$$ where the last equality follows from the change of variables $y=\varphi(x)$. After some elementary calculations, it follows that $$R(\varphi(\Omega),\varphi(\Sigma))[v]-R(\Omega,\Sigma)[u] = $$ $$\frac{\int_\Omega (|(D\varphi)^{-\top} \nabla u|^2 |\det D\varphi| - |\nabla u|^2)dx - \int_\Omega |\nabla u|^2 dx \bigg( \int_\Omega \frac{u^2 |\det D\varphi|}{d_{\varphi(\Sigma)}^2 \circ \varphi}dx -1 \bigg)}{\int_\Omega \frac{u^2 |\det D\varphi|}{d_{\varphi(\Sigma)}^2 \circ \varphi}dx}.$$

In order to get an estimate for the expression $$| (D\varphi)^{-\top} \nabla u  |^2 |\det D\varphi|-|\nabla u|^2,$$ we first note that $\| A^\top \| = \| A \|$ as operator norms. To get an upper bound for the operator norm of the inverse, we also make the assumption that $\varphi$ is a ``small'' diffeomorphism in the sense that $D\varphi(x) = I+\epsilon(x)$ where $\| \epsilon(x) \|<1$. In this case it is known that $$\| (D\varphi)^{-1}(x) \| \leq \frac{1}{1-\| \epsilon(x) \|}.$$ Besides, for such $\epsilon$ there is a constant $\kappa=\kappa(n)$ such that $$|\det (I+\epsilon)-1| \leq \kappa \| \epsilon \|,$$ so eventually we have the estimate $$| (D\varphi)^{-\top} \nabla u  |^2 |\det D\varphi|-|\nabla u|^2 \leq C |\nabla u|^2 \| D\varphi -I \|$$ for some constant $C>0$ provided that $\| D\varphi -I \|$ is small.

Next, for $x\in \Omega$, we obtain an estimate of $d_{\varphi(\Sigma)}(\varphi(x))$ in terms of $d_\Sigma(x)$. Since $d_\Sigma = d_{\bar{\Sigma}}$, we may assume that $\Sigma$ is closed. Then there exists $\sigma(x) \in \Sigma$ such that $d_\Sigma(x)=|x-\sigma(x)|$. Consider the straight line segment $\gamma:[0,1]\rightarrow \mathbb{R}^n$, $$\gamma(t)=(1-t)\sigma(x)+tx$$ joining these two points. Then clearly $d_\Sigma(x)=l(\gamma)$ (the arc length of $\gamma$). Then, by definition, we have that $$d_{\varphi(\Sigma)}(\varphi(x)) \leq l(\varphi \circ \gamma) = \int_0^1 |(\varphi\circ\gamma)'(t)|dt \leq \| D\varphi \|_{L^\infty(co(\Omega))} d_\Sigma(x),$$ thus $$d_{\varphi(\Sigma)}(\varphi(x)) \leq d_\Sigma(x)(1+\| D\varphi - I \|_{L^\infty(co(\Omega))}).$$

It follows that $$\int_\Omega \frac{u^2 |\det D\varphi|}{d_{\varphi(\Sigma)}^2 \circ \varphi}dx \geq  \frac{\inf_\Omega |\det D\varphi|}{(1+\| D\varphi-1 \|_{L^\infty(co(\Omega))})^2} \int_\Omega \frac{u^2}{d_\Sigma^2}dx$$ $$\geq \frac{1-\kappa \| D\varphi-I \|_{L^\infty(\Omega)}}{(1+\| D\varphi-1 \|_{L^\infty(co(\Omega))})^2},$$ the last inequality being valid due to normalisation, thus $$\int_\Omega \frac{u^2 |\det D\varphi|}{d_{\varphi(\Sigma)}^2 \circ \varphi}dx \geq 1-C \| D\varphi -I \|_{L^\infty(co(\Omega))}$$ for some constant $C$ provided that $\| D\varphi -I \|$ is small.

Using all these estimates we obtain $$R(\varphi(\Omega),\varphi(\Sigma))[v]-R(\Omega,\Sigma)[u] \leq cR(\Omega,\Sigma)[u] \| D\varphi-I \|_{L^\infty(co(\Omega))}.$$ for some $c>0$. Passing to the appropriate limit of minimisers, we get $$H(\varphi(\Omega),\varphi(\Sigma))-H(\Omega,\Sigma) \leq c H(\Omega,\Sigma) \| D\varphi-I \|_{L^\infty(co(\Omega))}$$ Replacing $\Omega$ and $\Sigma$ by $\varphi(\Omega)$ and $\varphi(\Sigma)$ and $\varphi$ by $\varphi^{-1}$, it follows that $$H(\Omega,\Sigma) - H(\varphi(\Omega),\varphi(\Sigma)) \leq cH(\varphi(\Omega),\varphi(\Sigma)) \| (D\varphi)^{-1}-I \|_{L^\infty(\varphi^{-1}(co(\varphi(\Omega))))}.$$ Since $$\| (D\varphi)^{-1}-I \| \leq \frac{\| D\varphi-I \|}{1-\| D\varphi-I \|},$$ it follows that there is $c>0$ such that the reverse inequality $$H(\Omega,\Sigma) - H(\varphi(\Omega),\varphi(\Sigma)) \leq cH(\Omega,\Sigma) \| D\varphi - I \|_{L^\infty(\varphi^{-1}(co(\varphi(\Omega))))}$$ also holds for small $\| D\varphi - I \|$. The result follows.
\end{proof}

For small $\| \varphi - id \|_{C^1}$, we have that if $\Omega$ is relatively compact, so is $co_\varphi(\Omega)$, so we immediately deduce the following.

\begin{corollary}
Let $\Omega\subset\mathbb{R}^n$ be open and bounded, and let $\Sigma \subset \partial \Omega$. Then the map $\varphi \longmapsto H(\varphi(\Omega),\varphi(\Sigma))$ is continuous with respect to the weak $C^1$ topology.
\end{corollary}

A few remarks are in order. First, the result does not hold for the case $k=0$ (homeomorphisms), as it is essential to be able to control first derivatives. Next, note that estimate \eqref{cont} holds independent of the boundedness of $\Omega$ or compactness of $\textrm{supp}(\varphi)$, and is therefore substantially more general than the corollary.

Although of no use to the sequel, we now present a collateral result that is obtained without extra effort. Instead of the standard Euclidean distance $\textrm{dist}(x,y)=|x-y|$, for $x,y\in \Omega$ one could use the alternative ``interior'' distance $$\widetilde{\textrm{dist}}(x,y) = \inf \{ l(\gamma): \gamma\in C^1([0,1],\Omega), \gamma(0)=x , \gamma(1)=y \},$$ and consider the Hardy problem 
\begin{equation}
\tilde{H}(\Omega,\Sigma) = \inf_{u\in H^1_0(\Omega)} \frac{\int_\Omega |\nabla u|^2 dx}{\int_\Omega u^2/\tilde{d}_\Sigma^2 dx},
\end{equation}
where $\tilde{d}_\Sigma(x)=\widetilde{\textrm{dist}}(x,\Sigma)$. For that case, we obtain the almost identical result

\begin{theorem}
Let $\Omega \subset \mathbb{R}^n$ be open set with non-empty boundary, and let $\Sigma \subset \partial \Omega$ be an arbitrary subset of the boundary. Then there exist $\epsilon>0$ and $c>0$ such that for every $C^1$ diffeomorphism $\varphi$ with $\|D\varphi-I \|<\epsilon$,
\begin{equation}
|\tilde{H}(\varphi(\Omega),\varphi(\Sigma)) - \tilde{H}(\Omega, \Sigma)| \leq c \tilde{H}(\Omega,\Sigma) \| D\varphi-I \|_{L^{\infty}(\Omega)}.
\end{equation}
\end{theorem}

\begin{proof}
The proof is almost identical to that of estimate \eqref{cont}. The only difference is that instead of picking $\gamma$ to be the straight line segment joining $x$ and $\sigma(x)$, one chooses a sequence of curves $\gamma_n$ such that $l(\gamma_n)\rightarrow \tilde{d}_\Sigma(x)$.
\end{proof}

Note that taking convex hulls is unnecessary here, since all distances are compared inside $\Omega$.

\section{Differentiability of the Hardy Constant}

Now we present our main results regarding differentiability. Our methodology is similar to the one developed in \cite{BL} (which concerns the case $\Sigma=\partial \Omega$), with appropriate modifications.

\begin{lemma}
Suppose that $\Omega \subset \mathbb{R}^n$ is open with non-empty boundary and let $\Sigma \subset \partial \Omega$ be closed. Let $\varphi \in \textrm{Diff}^1_c(\mathbb{R}^n)$, $\xi \in \mathfrak{X}^1_c(\mathbb{R}^n)$ and let $t_0>0$ be such that $$\varphi_t= \varphi+t\xi$$ is a $C^1$ diffeomorphism for all $t\in [-t_0,t_0]$. Then:
\begin{enumerate}
\item There exists a constant $c=c(\Omega,\varphi,\xi,t_0)$ such that
\begin{equation}
\label{dest}
|d_{\varphi_t(\Sigma)}^2(\varphi_t(x))-d_{\varphi(\Sigma)}^2(\varphi(x))| \leq c d_{\varphi(\Sigma)}^2(\varphi(x)) |t|
\end{equation}
for all $x\in \Omega$ and all $t\in [-t_0,t_0]$.
\item If $d_{\varphi(\Sigma)}$ is differentiable at $\varphi(x)$ and $\sigma(x) \in \Sigma$ is the single point such that $d_{\varphi(\Sigma)}(\varphi(x))=|\varphi(x)-\varphi(\sigma(x))|$, then
\begin{equation}
\label{dder}
\frac{d}{dt}\bigg|_{t=0} d_{\varphi_t(\Sigma)}^2(\varphi_t(x)) = 2(\varphi(x)-\varphi(\sigma(x)))\cdot(\xi(x)-\xi(\sigma(x))).
\end{equation}
\end{enumerate}
\end{lemma}

\begin{proof}
(1) Let $x\in\Omega$. Since $\Sigma$ is closed, there exists a $\sigma\in \Sigma$ such that $d_{\varphi(\Sigma)}(\varphi(x)) = |\varphi(x)-\varphi(\sigma)|$. It follows that $$d_{\varphi_t(\Sigma)}^2(\varphi_t(x)) \leq |\varphi_t(x)-\varphi_t(\sigma)|^2 = |\varphi(x)-\varphi(\sigma)+t(\xi(x)-\xi(\sigma))|^2$$ $$= d_{\varphi(\Sigma)}^2(\varphi(x))+2t(\varphi(x)-\varphi(\sigma))\cdot (\xi(x)-\xi(\sigma))+t^2|\xi(x)-\xi(\sigma)|^2.$$ Moreover, we have that $$|\xi(x)-\xi(\sigma)| = \bigg| \int_0^1 \frac{d}{ds} (\xi\circ\varphi^{-1})(s\varphi(\sigma)+(1-s)\varphi(x)) ds \bigg| $$ $$\leq \| D(\xi\circ\varphi^{-1}) \|_{L^\infty(co(\varphi(\Omega)))} |\varphi(x) - \varphi(\sigma)|$$ $$=  \| D(\xi\circ\varphi^{-1}) \|_{L^\infty(co(\varphi(\Omega)))} d_{\varphi(\Sigma)}(\varphi(x)).$$

Likewise, let $\sigma_t \in \Sigma$ be such that $d_{\varphi_t(\Sigma)}(\varphi_t(x))=|\varphi_t(x)-\varphi_t(\sigma_t)|$. Then $$d_{\varphi_t(\Sigma)}^2(\varphi_t(x))=|\varphi(x)-\varphi(\sigma_t)|^2+2t(\varphi(x)-\varphi(\sigma_t))\cdot (\xi(x)-\xi(\sigma_t))+t^2|\xi(x)-\xi(\sigma_t)|^2$$ $$\geq d_{\varphi(\Sigma)}^2(\varphi(x)) +2t(\varphi(x)-\varphi(\sigma_t))\cdot (\xi(x)-\xi(\sigma_t))+t^2|\xi(x)-\xi(\sigma_t)|^2,$$ and as before we have $$|\xi(x)-\xi(\sigma_t)| \leq \| D(\xi\circ\varphi_t^{-1}) \|_{L^\infty(co(\varphi_t(\Omega)))} d_{\varphi_t(\Sigma)}(\varphi_t(x)).$$

As $[-t_0,t_0]$ is compact, $\| D(\xi\circ\varphi_t^{-1}) \|_{L^\infty(co(\varphi_t(\Omega)))}$ attains a finite maximum value in it, and so follows the existence of a constant so that the conclusion holds.

(2) Assume that $d_{\varphi(\Sigma)}$ is differentiable at $\varphi(x)$. Thus there exists a unique $\sigma=\sigma(x) \in \Sigma$ such that $d_{\varphi(\Sigma)}(\varphi(x))=|\varphi(x)-\varphi(\sigma(x))|$. From (1), we know that
\begin{equation}
\label{dcont}
\lim_{t\rightarrow 0} d_{\varphi_t(\Sigma)}(\varphi_t(x)) = d_{\varphi(\Sigma)}(\varphi(x)).
\end{equation}

Now we claim that $\lim_{t\rightarrow 0} \sigma_t = \sigma$ ($\sigma_t$ as defined in the previous step). To this end, it suffices to show that $$\lim_{t\rightarrow 0} \varphi_t(\sigma_t) = \varphi(\sigma).$$ Assume, by contradiction, that there exists $\sigma'\in \Sigma$, $\sigma'\neq \sigma$, such that, possibly passing to a subsequence, $$\lim_{t\rightarrow 0} \varphi_t(\sigma_t) = \varphi(\sigma').$$ Then $$|\varphi(x)-\varphi(\sigma')|>d_{\varphi(\Sigma)}(\varphi(x))+\epsilon$$ for some $\epsilon>0$. In particular, $$\lim_{t\rightarrow 0} |\varphi_t(\sigma_t)-\varphi(x)|= |\varphi(\sigma')-\varphi(x)|>d_{\varphi(\Sigma)}(\varphi(x))+\epsilon.$$

Moreover, $$|\varphi_t(\sigma_t)-\varphi(x)|^2 = |\varphi_t(\sigma_t)-\varphi_t(x)+t\xi(x)|^2$$ $$=d_{\varphi_t(\Sigma)}^2(\varphi_t(x)) + 2t(\varphi_t(\sigma_t)-\varphi_t(x)) \cdot \xi(x)+t^2|\xi(x)|^2,$$ and by \eqref{dcont} we deduce that $$\lim_{t\rightarrow 0} |\varphi_t(\sigma_t)-\varphi(x)|=d_{\varphi(\Sigma)}(\varphi(x)),$$ a contradiction.

From the estimates of the previous step and the claim we deduce that $$\frac{d}{dt}\bigg|_{t=0} d_{\varphi_t(\Sigma)}^2(\varphi_t(x)) = 2(\varphi(x)-\varphi(\sigma(x)))\cdot(\xi(x)-\xi(\sigma(x))).$$
\end{proof}

From this point on, we will assume that $\Omega$ is bounded and Lipschitz. By the results of \cite{D}, we know that the Hardy inequality holds in $\Omega$ for some positive constant for $\Sigma = \partial \Omega$. Since $d_\Sigma \geq d_{\partial \Omega}$, the same is true if we choose any $\Sigma \subset \partial \Omega$.

\begin{lemma}
Suppose that $\Omega \subset \mathbb{R}^n$ is a bounded Lipschitz domain and let $\Sigma \subset \partial \Omega$ be closed. Let also $u\in H^1_0(\Omega)$ and $\rho \in L^\infty(\Omega)$. Then the function $G: Diff^1_c(\mathbb{R}^n) \rightarrow \mathbb{R}$ $(k\geq 1)$ given by $$G(\varphi)= \int_\Omega \frac{u^2 \rho}{d_{\varphi(\Sigma)}^2 \circ \varphi} dx$$  is Gateaux differentiable and, for $\xi \in \mathfrak{X}^1_c(\mathbb{R}^n)$ $$D_\varphi G(\xi)= -2 \int_\Omega \frac{u^2(x) \rho(x) (\varphi(x)-\varphi(\sigma(x))) \cdot (\xi(x)-\xi(\sigma(x)))}{d_{\varphi(\Sigma)}^4(\varphi(x))}dx.$$
\end{lemma}

\begin{proof}
Let $\varphi\in Diff^1_c(\mathbb{R}^n)$ and $\varphi_t=\varphi+t\xi$ as before. Then $$\frac{G(\varphi_t)-G(\varphi)}{t}= - \int_\Omega \frac{u^2 \rho (d_{\varphi_t(\Sigma)}^2\circ \varphi_t - d_{\varphi(\Sigma)}^2 \circ \varphi)}{t (d_{\varphi_t(\Sigma)}^2\circ \varphi_t)  (d_{\varphi(\Sigma)}^2 \circ \varphi)} dx.$$ By estimate \eqref{dest}, there is a constant $c>0$ such that $$\frac{u^2 \rho (d_{\varphi_t(\Sigma)}^2\circ \varphi_t - d_{\varphi(\Sigma)}^2 \circ \varphi)}{|t| (d_{\varphi_t(\Sigma)}^2\circ \varphi_t)  (d_{\varphi(\Sigma)}^2 \circ \varphi)} \leq c \frac{u^2 \rho}{d_{\varphi(\Sigma)}^2}$$ for $t$ sufficiently small. Since $\rho\in L^\infty(\omega)$ and $\Omega$ is bounded, and since $u\in H^1_0(\Omega)$ and the Hardy inequality holds (the later is true becauce $C^1$ diffeomorphisms preserve the Lipschitz property), it follows that the integrand is absolutely bounded by an $L^1$ function and the Dominated Convergence theorem applies.

Since $d_{\varphi(\Sigma)}(\varphi(x))$ is differentiable for almost all $x\in\Omega$, the unique point $\sigma(x)\in\Sigma$ is defined for almost all $x\in\Omega$ and the result follows by \eqref{dder}.
\end{proof}

We wish to prove that the Hardy constant $H(\varphi(\Omega),\varphi(\Sigma))$ is Gateaux differentiable with respect to $\varphi$, which is equivalent to proving that the map $t\mapsto H(\varphi_t(\Omega),\varphi_t(\Sigma))$ is differentiable with respect to $t$ for any $\xi \in \mathfrak{X}^1_c(\mathbb{R}^n)$, where $$\varphi_t=\varphi+t\xi.$$ Doing so will be possible provided that there are actual minimisers to the constants $H(\varphi_t(\Omega),\varphi_t(\Sigma))$, and that these actually behave ``well'' as $t$ varies, i.e. they are stable.

Here we draw some important facts coming from other works that are vital in order to proceed.

\begin{lemma}
Suppose that $\Omega \subset \mathbb{R}^n$ $(n\geq 2)$ is a smooth bounded domain, and let $\Sigma \subset \partial \Omega$ be a closed submanifold of dimension $s\in \{0,1,...,n-1 \}$. Consider the Hardy problem
\begin{equation}
\label{Hardy}
H(\Omega,\Sigma)= \inf_{\substack{u\in H^1_0(\Omega), \\ u\neq 0}} \frac{\int_\Omega |\nabla u|^2 dx}{\int_\Omega u^2/d_\Sigma^2 dx}.
\end{equation}
Then precisely one of the following is true:
\begin{enumerate}
\item The problem has a minimiser and $H(\Omega,\Sigma)<(n-s)^2/4$.
\item The problem does not have a minimiser and $H(\Omega,\Sigma)=(n-s)^2/4$.
\end{enumerate}
\end{lemma}

\begin{proof}
This is Corollary 1.3 in \cite{FM}. The case $s=0$ was treated separately in \cite{FMu}, and the case $s=n-1$ is well known (see \cite{MMP}).
\end{proof}

So in order to proceed we need from now on the additional assumption that $H(\varphi(\Omega), \varphi(\Sigma)) < (n-s)^2/4$ in order to guarantee the existence of minimisers. This assumption is not terribly restrictive, since $\varphi \mapsto H(\varphi(\Omega),\varphi(\Sigma))$ is a continuous map and the inverse image of $(-\epsilon,(n-s)^2/4)$ with respect to that map is an open set of $Diff^1_c(\mathbb{R}^n)$.

Next we provide some estimates for these minimisers.

\begin{lemma}
Let $\Omega$ and $\Sigma$ be as in the previous lemma, and suppose that $v\in H^1_0(\Omega)$ is a minimiser of $\eqref{Hardy}$. Then there is a constant $C=C(\Omega,\Sigma)>0$ such that $$v < C d_{\partial \Omega} d_{\Sigma}^\alpha,$$ where $$\alpha = \frac{s-n+ \sqrt{(n-s)^2 - 4H(\Omega,\Sigma)}}{2}$$
\end{lemma}

\begin{proof}
This was proven in \cite{MN} for the eigenfunction corresponding to the first eigenvalue of the relevant Schrödinger operator (Lemmas 2.1 and 2.2). The same steps can be repeated for $\lambda=0$, which simplifies the proof even further.
\end{proof}

\begin{theorem}
Suppose that $\Omega \subset \mathbb{R}^n$ $(n\geq 2)$ is a smooth bounded domain, and let $\Sigma \subset \partial \Omega$ be a closed submanifold of dimension $s$. Let $v\in H^1_0(\Omega)$ be a minimiser of $\eqref{Hardy}$ (and so $H(\Omega,\Sigma)<(n-s)^2/4$). Then the following estimates are satisfied: $$v \leq C d_\Sigma^{\alpha+1},$$ $$|\nabla v| \leq C d_\Sigma^\alpha,$$ where $C=C(\Omega,\Sigma)$.
\end{theorem}

\begin{proof}
The first estimate is obvious from the previous lemma and the fact that $d_{\partial \Omega} \leq d_\Sigma$.

For the second one we proceed as follows. Let $x\in \Omega$ and let $R=d_{\partial \Omega}(x)/3$. Then for every $y\in B(x,R)$ we have that $$2R\leq d_{\partial \Omega}(y) \leq 4R.$$ At this point we invoke a gradient estimate such as $$|\nabla v(x)| \leq C(n) \bigg( \frac{1}{R} \sup_{\partial B(x,R)} |v| + R \sup_{B(x,R)}|f| \bigg),$$ where $f=H(\Omega,\Sigma)v/d_\Sigma^2$, see for example \cite{GT} (paragraph 3.4) for an analogue with cubes. Thus, after some elementary calculations, we get $$|\nabla v (x)| \leq C \sup_{B(x,R)} d_\Sigma^\alpha \leq C (d_\Sigma(x)+R)^\alpha \leq C d_\Sigma^\alpha(x),$$ where in each step constant factors are absorbed in $C$.
\end{proof}

\begin{theorem}
Let $\Omega \subset \mathbb{R}^n$ $(n\geq 2)$ be a smooth bounded domain, and let $\Sigma \subset \partial \Omega$ be a closed submanifold of dimension $s$. Suppose that $H(\Omega,\Sigma)<(n-s)^2/4$. Thus for any $\xi \in \mathfrak{X}^1_c(\mathbb{R}^n)$, $\varphi_t=id+t\xi \in Diff^1_c(\mathbb{R}^n)$ and $H(\varphi_t(\Omega),\varphi_t(\Sigma))<(n-s)^2/4$ for $t$ small enough.

Let $v_t$ be a one-parameter family of positive minimisers for $H(\varphi_t(\Omega),\varphi_t(\Sigma))$, normalised by $$\int_{\varphi_t(\Omega)} \frac{v_t^2}{d_{\varphi_t(\Sigma)}^2}dx = 1,$$ and let $u_t = v_t \circ \varphi_t: \Omega \rightarrow \mathbb{R}$. Then
\begin{equation}
u_t \rightarrow u_0 \textrm{ in } H^1_0(\Omega).
\end{equation}
\end{theorem}

\begin{proof}
By the normalisation condition on the minimisers, it follows that $\| \nabla v_t \|_{L^2(\varphi_t(\Omega))}=H(\varphi_t(\Omega),\varphi_t(\Sigma))$, thus $\| v_t \|_{H^1_0(\varphi(\Omega))}$ and $\| u_t \|_{H^1_0(\Omega)}$ are uniformly bounded. Hence, possibly passing to a subsequence, there is a $\tilde{u}_0$ such that $$u_t \rightarrow \tilde{u}_0 \textrm{ weakly in } H^1_0,$$ $$u_t \rightarrow \tilde{u}_0 \textrm{ in } L^2.$$

We will show that $\tilde{u}_0$ satisfies the same normalisation condition, i.e. $$\int_\Omega \frac{\tilde{u}_0^2}{d_\Sigma^2}dx =1.$$ This is actually a consequence of the DCT applied on $$\int_\Omega \frac{u_t^2}{d_{\varphi_t(\Sigma)}^2(\varphi_t(x))} |\det D\varphi_t(x)|dx=1,$$ provided it is applicable. Indeed, from the previous estimates, we have that there are $C>0$ and $\alpha$ such that $2\alpha+n-s>0$ such that $$u_t(x)\leq C d_{\varphi_t(\Sigma)}^{\alpha+1}(\varphi_t(x))$$ uniformly in $t$ for $t$ small enough. It follows that $$\frac{u_t^2}{d_{\varphi_t(\Sigma)}^2(\varphi_t(x))} |\det D\varphi_t(x)| \leq C d_{\varphi_t(\Sigma)}^{2\alpha}(\varphi_t(x)).$$ Choosing $\Sigma_\epsilon(t) = \{x\in \Omega: d_{\varphi_t(\Sigma)}(\varphi_t(x)) < \epsilon \}$ and passing to exponential coordinates such that $r(x)=d_{\varphi_t(\Sigma)}(\varphi_t(x))$, we have that $$\int_{\Sigma_\epsilon(t)} d_{\varphi_t(\Sigma)}^{2\alpha}(\varphi_t(x))dx \leq C \int_0^\epsilon r^{2\alpha+n-s-1}dr,$$ where the integral of the RHS is convergent since $2\alpha+n-s-1>-1$. Hence the integrand is uniformly bounded in $t$ by an integrable function, and the claim follows.

From vector inequality $|a|^2 \geq |b|^2+2b\cdot (a-b)$, it follows that $$H(\varphi_t(\Omega),\varphi_t(\Sigma)) = \int_\Omega |(D\varphi_t)^{-\top} \nabla u_t|^2 |\det D\varphi_t|dx \geq$$ $$\int_\Omega |\nabla \tilde{u}_0|^2 |\det D\varphi_t|dx+2\int_\Omega \nabla \tilde{u}_0 \cdot ((D\varphi_t)^{-\top} \nabla u_t - \nabla \tilde{u}_0) |\det D\varphi_t|dx.$$ By the DCT and the continuity of $H$, it follows that $$H(\Omega,\Sigma)\geq \int_\Omega |\nabla \tilde{u}_0|^2 dx,$$ so $\tilde{u}_0$ must be a positive normalised minimiser, and by the uniqueness of such minimisers it follows that $\tilde{u}_0 = u_0$.

Moreover, also by the DCT, we have that $$\lim_{t\rightarrow 0} \bigg( H(\varphi_t(\Omega),\varphi_t(\Sigma))-\int_\Omega |\nabla u_t|^2 dx \bigg) =$$ $$\lim_{t\rightarrow 0} \int_\Omega \bigg( |(D\varphi_t)^{-\top} \nabla u_t|^2 |\det D\varphi_t| - |\nabla u_t|^2 \bigg) dx =0,$$ so it follows that $$\lim_{t\rightarrow 0} \int_\Omega |\nabla u_t|^2 dx = H(\Omega,\Sigma) = \int_\Omega |\nabla u_0|dx.$$ Since weak convergence and convergence in norm imply strong convergence, the proof is complete.
\end{proof}

\begin{theorem}
Let $\Omega \subset \mathbb{R}^n$ $(n\geq 2)$ be a smooth bounded domain, and let $\Sigma \subset \partial \Omega$ be a closed submanifold of dimension $s$. Suppose that $H(\Omega,\Sigma)<(n-s)^2/4$ and let $v$ be a minimiser that achieves $H(\Omega,\Sigma)$ normalised by $$\int_{\Omega} \frac{v^2}{d_{\Sigma}^2}dx = 1.$$ Then the map $H: \varphi \mapsto H(\varphi(\Omega),\varphi(\Sigma))$ is differentiable at $id_{\mathbb{R}^n}$ and $$D_{id}H(\xi) = \int_\Omega [ |\nabla v|^2 div(\xi) -2 (D\xi)\nabla v \cdot \nabla v ]dx$$ $$+H(\Omega,\Sigma) \int_\Omega \bigg[2 \frac{v^2}{d^3_\Sigma} \nabla d_\Sigma \cdot (\xi-\xi\circ \sigma) - \frac{v^2}{d_\Sigma^2} div(\xi) \bigg] dx,$$ where $\sigma(x)$ is the (a.e unique) point in $\Sigma$ such that $d_\Sigma(x)=|x-\sigma(x)|$.
\end{theorem}

\begin{proof}
Let $\varphi_t = id + t\xi$ and $v_t$ a sequence of positive normalised minimisers as before. By the definition of the Hardy constant and change of variables, we have that $$H(\Omega,\Sigma)= \min_{u\in H^1_0\setminus \{ 0 \} } R_t[u],$$ where $R_t[u]=N_t[u]/D_t[u]$, $$N_t[u] = \int_\Omega |(D\varphi_t)^{-\top} \nabla u|^2 |\det D\varphi_t|dx,$$ $$D_t[u] = \int_\Omega \frac{u^2}{d_{\varphi_t(\Sigma)}\circ \varphi_t} |\det D\varphi_t|dx.$$ Since $v_t$ achieves $H(\varphi_t(\Omega),\varphi_t(\Sigma))$, we have that $H(\varphi_t(\Omega),\varphi_t(\Sigma)) = R_t[u_t]$, where $u_t=v_t\circ \varphi_t$ as before.

It follows, by the definition of the Hardy constant, that $$R_t[u_t]-R_0[u_t]\leq H(\varphi_t(\Omega),\varphi_t(\Sigma)) - H(\Omega,\Sigma) \leq R_t[u_0] - R_0[u_0].$$ Now, $R_t[u]$ is a function of two arguments, a real number $t$ and a function $u$. The partial derivative of this function with respect to $t$ is denoted by $R_t'[u]$. The last inequality together with the mean value theorem on the first argument of $R$ imply that there are numbers $\xi(t)$ and $\eta(t)$ such that $|\xi(t)|,|\eta(t)| < |t|$ and $$R_{\xi(t)}'[u_t]t \leq H(\varphi_t(\Omega),\varphi_t(\Sigma)) - H(\Omega,\Sigma) \leq R_{\eta(t)}'[u_0].$$

If we show that $R_{\xi(t)}'[u_t]t$ and $R_{\eta(t)}'[u_0]$ converge to the same number as $t\rightarrow 0$, differentiability at $t=0$ is established. Some basic calculations reveal that $$\frac{d}{dt} |(D\varphi)^{-\top} \nabla u |^2 = -2 (D\varphi_t)^{-1} D\xi (D\varphi_t)^{-1} (D\varphi_t)^{-\top} \nabla u \cdot \nabla u,$$ $$\frac{d}{dt} |\det D\varphi_t| = \frac{div(\xi)}{|\det D\varphi_t^{-1}\circ \varphi_t|}.$$ It follows that $$N_t'[u] = \int_\Omega |(D\varphi_t)^{-\top} \nabla u|^2 \frac{div(\xi)}{|\det D\varphi_t^{-1}\circ \varphi_t|} dx $$ $$-2 \int_\Omega (D\varphi_t)^{-1} D\xi (D\varphi_t)^{-1} (D\varphi_t)^{-\top} \nabla u \cdot \nabla u |\det D\varphi_t|dx,$$ and $$D_t'[u] = \int_\Omega \frac{u^2}{d_{\varphi_t(\Sigma)}\circ \varphi_t} \frac{div(\xi)}{|\det D\varphi_t^{-1}\circ \varphi_t|} dx$$ $$-2\int_\Omega \frac{u^2 \nabla d_{\varphi_t(\Sigma)}\circ \varphi_t \cdot (\xi - \xi\circ \sigma_t)}{d_{\varphi_t(\Sigma)}\circ \varphi_t} |\det D\varphi_t|dx.$$ By DCT, it follows that $$\lim_{t\rightarrow 0} R_{\eta(t)}'[u_0] = R_0'[u_0],$$ and by DCT together with the previous stability result, we also have $$\lim_{t\rightarrow 0} R_{\xi(t)}'[u_t] = R_0'[u_0]$$ and the claim is proven.

It remains to compute the derivative. We have $$\frac{d}{dt} \bigg|_{t=0} H(\varphi_t(\Omega),\varphi_t(\Sigma)) = \frac{N_0'[u_0]D_0[u_0]-N_0[u_0]D_0'[u_0]}{D_0^2[u_0]}$$ $$=N_0'[u_0] - H(\Omega,\Sigma) D_0'[u],$$ the last equality being valid due to normalisation. The result immediately follows from the previous calculations, putting $t=0$ and taking into account that $u_0=v$.
\end{proof}

\section{Differentiability with respect to boundary diffeomorphisms}

Finally, we turn our attention to the matter of differentiability of the map $$\varphi \longmapsto H(\Omega,\varphi(\Sigma)),$$ for $\varphi\in Diff^1(\partial \Omega)$. Note that in the case where $s=n-1$, this problem is irrelevant since the boundary as a whole remains invariant under boundary diffeomorphisms, so in this sense it is new.

First we establish a continuity result. In particular, if $\varphi\in Diff^1(\partial \Omega)$, the map $\varphi \mapsto H(\Omega, \varphi(\Sigma))$ is shown to be continuous with respect to the $C^1$ topology.

\begin{theorem}
Let $\Omega \subset \mathbb{R}^n$ be a bounded open set with smooth non-empty boundary, and let $\Sigma \subset \partial \Omega$ be an arbitrary subset of the boundary. Then there exist $\epsilon>0$ and $c>0$ such that for any $\varphi\in Diff^1(\partial \Omega)$ satisfying $\| \varphi - id \|_{C^1(\partial \Omega)} <\epsilon$, the estimate
\begin{equation}
|H(\Omega,\varphi(\Sigma))-H(\Omega,\Sigma)|\leq c H(\Omega,\Sigma) \| \varphi - Id \|_{C^1(\partial\Omega)}
\end{equation}
holds.
\end{theorem}

\begin{proof}
This can actually be reduced to the first case. One simply needs to extend diffeomorphisms of the boundary to diffeomorphisms of the ambient space. This cannot be done for an arbitrary diffeomorphism, but for small diffeomorphisms it is achievable since $Diff^1(\partial \Omega)$ is locally contractible. Indeed, for $\| \varphi - id \|_{C^0}< inj(\partial \Omega)$ (the injectivity radious of $\partial \Omega$ is a positive number since $\partial \Omega$ is compact), define a homotopy $h:\partial \Omega \times [0,1] \rightarrow \partial \Omega$, $$h(x,t)=\exp_x(t\exp_x^{-1}(\varphi(x))),$$ where $\exp$ stands for the exponential map of $\partial \Omega$ as a Riemannian submanifold of $\mathbb{R}^n$, while the assumption above ensures that $h(\cdot,t)$ remains a diffeomorphism for all $t$.

We now pick a neighbourhood of $\partial \Omega$ that is diffeomorphic to $\partial \Omega \times (-\epsilon,\epsilon)$, and a cut-off function $f:(-\epsilon,\epsilon)\rightarrow \mathbb{R}$ that is $1$ in a neighbourhood of $0$. Then $$\Phi(x,y) = h(x,1-f(y))$$ is a diffeomorphism of $\mathbb{R}^n$ with compact support that extends $\varphi$ (extend trivially outside the neighbourhood by the identity). We then apply \eqref{cont} for $\Phi$ and the result follows from the fact that $$\| \Phi - Id_{\mathbb{R}^n} \|_{C^1} \leq c \| \varphi - Id_{\partial \Omega} \|_{C^1},$$ which is obvious by the construction.
\end{proof}

Similar to the Euclidean case, $Diff^k(\partial\Omega)$ has a differential structure that is locally homeomorphic to the Banach space $\mathfrak{X}^k(\partial\Omega)$ (note that here we need not take vector fields with compact support since $\partial\Omega$ is by assumption compact). The differential of a map $h:Diff^k(\partial\Omega) \rightarrow \mathbb{R}$ at $\varphi \in Diff^k(\partial\Omega)$ along $\xi \in \mathfrak{X}^k(\partial\Omega)$ is given by $$D_\varphi h(\xi) = \frac{d}{dt}\bigg|_{t=0} h(\exp(t\xi) \circ \varphi),$$ provided that the limit exists, where $\exp(t\xi) \in Diff^k(\partial\Omega)$ is the map obtained by exponential mapping along $\xi$, which is always a diffeomorphism for $t$ small enough due to compactness.

We finally show that the Hardy constant is Gateaux differentiable with respect to such boundary diffeomorphisms.

\begin{theorem}
Let $\Omega \subset \mathbb{R}^n$ be bounded and of smooth boundary. Then the map $h:Diff^1(\partial\Omega) \rightarrow \mathbb{R}$, $\varphi \mapsto H(\Omega,\varphi(\Sigma))$ is differentiable at all points where $H(\Omega,\varphi(\Sigma))<(n-s)^2/4$.
\end{theorem}

\begin{proof}
Without loss of generality, let $\varphi=id_{\partial \Omega}$, and let $\xi\in \mathfrak{X}^1(\partial\Omega)$. Then one can extend $\xi$ to a $\Xi\in \mathfrak{X}^1_c(\mathbb{R}^n)$ (using a standard argument involving partitions of unity, for example). One can also assume that the support of $\Xi$ lies within a neigbourhood of the form $\partial\Omega\times (-\epsilon,\epsilon)$, equipped with a metric such that $\partial \Omega$ is a totally geodesic submanifold. Then we have that $$D_{id_{\partial \Omega}} h(\xi) = \frac{d}{dt}\bigg|_{t=0} h(\exp(t\xi))= \frac{d}{dt}\bigg|_{t=0} H(\Omega,\exp(t\xi)(\Sigma)).$$ Since $\exp(t\xi)(\Omega)=\Omega$ and $\exp(t\Xi) |_{\partial\Omega}=\exp(t\xi)$, it follows that $$D_{id_{\partial \Omega}} h(\xi) = \frac{d}{dt}\bigg|_{t=0} H(\exp(t\Xi)(\Omega), \exp(t\xi)(\Sigma))$$ $$= D_{id_{\mathbb{R}^n}}H(\frac{d}{dt}\bigg|_{t=0} \exp(t\Xi) )= D_{id_{\mathbb{R}^n}}H(\Xi),$$ where in the last equalities we regard $H$ as the function $\varphi\mapsto H(\varphi(\Omega),\varphi(\Sigma))$ as discussed in the previous section.
\end{proof}

There is a particularly neat way to express this form of differentiability in the special case $s=0$ (a point boundary singularity).

\begin{corollary}
Let $\Omega \subset \mathbb{R}^n$ be bounded and of smooth boundary. Then the map $H:\partial\Omega \rightarrow \mathbb{R}$, $\sigma \mapsto H(\Omega,\{\sigma\})$ is differentiable at every $\sigma\in \partial \Omega$ where $H(\Omega,\{ \sigma \})<(n-s)^2/4$.
\end{corollary}

\paragraph{\textbf{Acknowledgements}} Special thanks are owed to my PhD supervisor, Professor G. Barbatis, for the time he spent reviewing the article and offering useful suggestions. This research was supported by the Hellenic Foundation for Research and Innovation (HFRI) under the HFRI PhD Fellowship grant (Fellowship Number 1250).




\bibliographystyle{model1-num-names}
\appendix
\section*{Bibliography}
\bibliography{sample.bib}



\end{document}